\documentclass{article}
\usepackage{amsfonts}
\usepackage{amssymb,amsmath}
\usepackage{verbatim}

\newtheorem{theorem}{Theorem}[section]

\newtheorem{remark}{Remark}[section]
\newtheorem{lemma}{Lemma}[section]
\numberwithin{equation}{section}
\input{tcilatex}

\begin{document}

\title{\textbf{\LARGE On necessary and sufficient conditions for the
variable exponent Hardy type inequality} \\
[15pt]}
\author{Farman I. Mamedov \\
\emph{Institute Mathematics and Mechanics of Nat.Acad.Sci., Azerbaijan} \\
{\small e-mail: mfarman@math.ab.az}}
\date{ December 20, 2012\\
[18pt]}
\maketitle

\begin{abstract}
{{\small {\ We derive a number of equivalent criterions for the variable exponent Hardy type inequality
\begin{equation*}
\left\Vert \frac{1}{x}\int_{0}^{x}f(t)dt\right\Vert _{L^{p(.)}(0,1)}\leq
C\left\Vert f\right\Vert _{L^{p(.)}(0,1)}; f\geq 0.
\end{equation*}
to hold, whenever the exponent $p:(0,1)\to (1,\infty)$ is increasing or decreasing near small neighborhood of the origin.}}}
\end{abstract}

\bigskip

\bigskip

\bigskip

\textbf{{\footnotesize {Key words and phrases :}}} {\footnotesize {Hardy
operator, Hardy type inequality, variable exponent, weighted inequality,
necessary and sufficient condition.}}

{\footnotesize {\textbf{2000 Mathematical Subject Classification:} 42A05,
42B25, 26D10, 35A23}\\[25pt]
}

\section{Introduction}

\label{S:1}

\quad We study Hardy's inequality
\begin{equation}  \label{E:1.1}
\ \ \left\Vert x^{-1}Hf\right\Vert _{L^{p\left( .\right) }\left( 0,1\right)
}\leq C\left\Vert f\right\Vert _{L^{p\left( .\right) }\left( 0,1\right) }
\end{equation}%
in the norms of variable exponent Lebesgue space $L^{p\left(.\right)}(0,1).$
Here $Hf(x)=\int_{0}^{x}f(t)dt$ is Hardy's operator and the constant $C>0$
does not depend on arbitrary positive measurable function $f.$ This subject has been studied
by several authors (see, e.g. \cite{[2]}, \cite%
{[4]}, \cite{[5]}, \cite{[7]}, \cite{[8]}, \cite{[9]}, \cite{[10]}, \cite%
{[11]}, \cite{[12]}, \cite{[13]}, \cite{[14]}, \cite{[15]}, \cite{[16]}).

There are several sufficient conditions on the function $p:(0,1)\rightarrow
(1,\infty )$ for the inequality \eqref{E:1.1} to hold. They are expressed in
terms of regularity conditions for $p\,$ at the origin. It follows from the
results of works \cite{[4]}, \cite{[9]}, \cite{[14]}  ( see, also \cite{[2]}, \cite {[11]}, \cite{[13]})
that the inequality \eqref{E:1.1} holds if $p^{-}=\inf $ $p>1$, $p^{+}=\sup $
$p(x)<\infty $ and the condition
\begin{equation}
A:=\underset{x\rightarrow 0}{\limsup }\,|p(x)-p(0)|\log \frac{1}{x}<\infty .
\label{E:1.2}
\end{equation}%
is satisfied.

One can think that the inequality \eqref{E:1.1} does not need for a
condition type of \eqref{E:1.2} at all. Since there exists an example of
function $p$ for which the inequality \eqref{E:1.1} is violated by some
sequence of functions $\{f_{k}\}$ (see, \cite{[9]}, \cite{[7]}), we see that
the inequality \eqref{E:1.1} does not hold without restriction on $p$ (Note,
the $p$ there is not monotone and does not satisfy \eqref{E:1.2}).
In \cite{[10]} (see, also \cite{[7]}), we had proved that the condition
\begin{equation}
B:=\underset{x\to 0}{\lim \sup }\,\left[ p(x)-p\left( \frac{x}{2}\right)\right] %
\log \frac{1}{x}<\infty  \label{E:1.3}
\end{equation}%
is necessary for this case. Note that, condition \eqref{E:1.3} is strictly
weaker than \eqref{E:1.2}. This condition is new and somewhat surprising.
For example, it is satisfied by $p(x)=p(0)+\frac{C}{\left( \ln \frac{1}{x}%
\right) ^{\alpha }}$ \thinspace\ and $0<\alpha<1,$ \thinspace\ $C>0,$ whereas
the condition \eqref{E:1.2} is not satisfied. For the exponent, that is
nondecreasing near the origin, the condition \eqref{E:1.3} is also
sufficient if the number $B$ satisfies $B<p(0)\left(p(0)-1\right)$ (see,
\cite{[10]}). Unfortunately, the good condition \eqref{E:1.3} is no longer
sufficient for the inequality \eqref{E:1.1} to hold if the condition on $B$
be ignored. In this case, a necessary and sufficient condition is still an
open problem.

In Theorem \ref{t:2.2}, we prove that the condition
\begin{equation}  \label{E:1.4}
\int_{a}^{1}\left( a^{\frac{1}{p^{\prime }(a)}} x^{-\frac{1}{p^{\prime }(x)}%
}\right) ^{p(x)}\frac{dx}{x}\leq C,\text{ \ }0<a<1
\end{equation}
and several other equivalent conditions are necessary and sufficient for the
inequality \eqref{E:1.1} to hold in the case of nondecreasing exponents.

Also, in Theorem \ref{t:2.1}, we prove that no condition is needed if the
exponent $p$ is nonincreasing at small neighborhood of the origin.

We refer to the monograph \cite{[3]} and references therein for a full
description of variable exponent Lebesgue spaces and boundedness of
classical integral operators there.

\section{Main results and notation}\label{S:2}

\qquad As to the basic properties of spaces\ \ $L^{p\left( .\right) }$ $,$
we refer to \cite{[6]}, \cite{[17]}. Throughout this paper, it is assumed
that\ $p\left( x\right) $ is a measurable function in $\left( 0,1\right) ,$
taking its values from the interval $\left[ 1,\infty \right) $ \ with \ $%
p^{+}=\sup \left\{ p\left( x\right) :x\in \left( 0,1\right) \right\} <\infty
$ . The space of functions\ $L^{p\left( .\right) }\left( 0,1\right) $ is
introduced as the class of measurable functions \ $f\left( x\right) $ on\ \ $%
\left( 0,1\right) $ which have a finite $I_{p\left( .\right) }\left(
f\right) =\int_{0}^{1}\left\vert f\right\vert ^{p\left( x\right) }dx$
modular. A norm in\ $L^{p\left( .\right) }\left( 0,1\right) $ is given in
the form
\begin{equation*}
\left\Vert f\right\Vert =\left\{ \lambda >0:I_{p\left( .\right) }\left(
\frac{f}{\lambda }\right) \leq 1\right\} .
\end{equation*}%
For $\ 1<p^{-},$ $p^{+}<\infty $ the space $L^{p(.)}(0,1)$ is a reflexive
Banach space.

The relation between modular and norm is expressed by the following
inequalities (see, f.e. \cite{[17]}):
\begin{equation}
\left\Vert f\right\Vert _{L^{p(.)}(0,l)}^{p^{+}}\leq I_{p}\left( f\right)
\leq \left\Vert f\right\Vert _{L^{p(.)}(0,l)}^{p^{-}},\,\ \ 1\geq
\,\left\Vert f\right\Vert _{p(.),}  \label{e:2.1}
\end{equation}

\begin{equation}  \label{e:2.2}
\left\Vert f\right\Vert _{L^{p(.)}(0,l)}^{p^{-}} \leq I_{p}\left( f\right)
\leq \left\Vert f\right\Vert _{L^{p(.)}(0,l)}^{p^{+}},\,\,\,\,1\leq
\left\Vert f\right\Vert _{p(.)} .
\end{equation}
Such estimates alow us to perform our estimates in terms of a modular.

For the function $1\leq p(x)<\infty $ $\ \ \ \ $\ $p^{\prime }(x)$ denotes
the conjugate function of $p(x),$ $\frac{1}{p(x)}+\frac{1}{p^{\prime }(x)}=1$
\ and $p^{\prime }=\infty $ \ if \ $p=1.$ We denote by $\ C,C_{1},C_{2},...$
various positive constants whose values may vary at each appearance. By $\chi _{E}$ we
denote the characteristic function of set $E.$
We say the function $f$ is
almost increasing (almost decreasing) on $[0,1]$ if $f(x)\leq Cf(y)$ \, ($f(y)\leq Cf(x)$) for all $x\leq y$  in $ [0,1]$\, \, and $C>0.$

Following main results are obtained in this paper. \vspace{2pt}

\begin{theorem}
\label{t:2.1} Let $p:(0,1)\to [1,\infty)$ be a measurable function such that
$p$ is nonincreasing on some interval $(0,\epsilon), \,
\epsilon>0$ and $p^{+}<\infty.$ Then it holds the inequality \eqref{E:1.1} for any positive
measurable function $f.$
\end{theorem}

\vspace{2pt}

\begin{theorem}
\label{t:2.2} Let $p:(0,1)\rightarrow \lbrack 1,\infty )$ be a nondecreasing
function such that $p(1)<\infty .$ Then the following statements are
equivalent:
\begin{itemize}
\item[1.] There exists a constant $C>0$ such that the inequality
\begin{equation}
\ \ \left\Vert x^{-1}Hf\right\Vert _{L^{p\left( .\right) }\left( 0,1\right)
}\leq C\left\Vert f\right\Vert _{L^{p\left( .\right) }\left( 0,1\right) }
\label{E:2.1}
\end{equation}
holds for any positive measurable function $f.$

\item[2.] The condition
\begin{equation}  \label{E:2.2}
\int_{a}^{1} x^{-\frac{1}{p^{\prime }(x)}}\frac{dx}{x}\leq Ca^{-\frac{1}{%
p^{\prime }(a)}}, \, 0<a<1
\end{equation}
is satisfied.

\item[3.] There exists an $\epsilon>0$ such that the function $x^{-\frac{1}{p^{\prime }(x)}+\epsilon}$
is almost decreasing:
\begin{equation}  \label{e:2.5}
t_{2}^{-\frac{1}{p^{\prime }(t_{2})}+\epsilon}\leq C t_{1}^{-\frac{1}{p^{\prime }(t_{1})}+\epsilon} \hspace{0.5cm}\text{as} \hspace{0.5cm}
0<t_{1}\leq t_{2}<1.
\end{equation}
\item[4.] The condition \eqref{E:1.4}
is satisfied.
\item[5.] The condition
\begin{equation}  \label{E:2.4}
\Vert x^{-1} \Vert_{p(.);(a,1)}\leq Ca^{-\frac{1}{p^{\prime }(a)}}, \, \, \,
0<a<1.
\end{equation}
is satisfied.
\end{itemize}
\end{theorem}

\section{ Proof of Theorem \protect\ref{t:2.1}.}

\label{S:3}

\qquad Let $f(x)\geq 0$ be a measurable function such that $\left\Vert
f\right\Vert _{L^{p(.)}(0,1)}\leq 1.$ Then it follows from the inequality %
\eqref{e:2.1} that
$
I_{p(.)}\left( f\right) \leq 1.
$
In order to prove Theorem \ref{t:2.1} we have to show that%
\begin{equation}  \label{e:3.2}
\left\Vert x^{-1}Hf\right\Vert _{L^{p(.)}(0,1)}\leq C_{1}.
\end{equation}%
To prove \eqref{e:3.2}, we establish the estimate
\begin{equation*}
I_{p(.)}\left( \frac{Hf}{x}\right) \leq C_{2}.
\end{equation*}%

Using triangle inequality for $p(.)$-norms and $\epsilon \in (0,1),$ we have%
\begin{equation*}
\left\Vert x^{-1}Hf\right\Vert _{L^{p(.)}(0,1)}\leq \left\Vert
x^{-1}Hf\right\Vert _{L^{p(.)}(0,\epsilon )}+\left\Vert x^{-1}Hf\right\Vert
_{L^{p(.)}(\epsilon ,1)}
\end{equation*}%
\begin{equation}
:=i_{1}+i_{2}.  \label{e:3.4}
\end{equation}%
Taking into account
\begin{equation*}
\frac{Hf(x)}{x}=\int_{0}^{1}f(tx)dt
\end{equation*}%
and using Minkowskii's inequality for $L^{p(.)}$ norms, it follows that
(see, \cite{[6]}, \cite{[17]})
\begin{equation}
i_{1}=\left\Vert \frac{Hf}{x}\right\Vert _{p(.);\,(0,\epsilon )}\leq
\left\Vert \int_{0}^{1}f(.\,t)dt\right\Vert _{p(.);\,(0,\epsilon )}\leq
\int_{0}^{1}\left\Vert f(.\,t)\right\Vert _{p(.);\,(0,\epsilon )}dt.
\label{E:3.7}
\end{equation}

Let us estimate the term $\left\Vert f(.\,t)\right\Vert _{p(.);\,(0,\epsilon
)}$ for $0<t<1.$ Since $p$ is nonincreasing on $(0,\epsilon ),$ we have $%
p(x)\leq p(tx)$ \thinspace\ for \thinspace\ $x\in (0,\epsilon ).$ Therefore,
\begin{align*}
& \int_{0}^{\epsilon }f(xt)^{p(x)}dx\leq \int_{0}^{\epsilon
}f(xt)^{p(x)}\chi _{f(xt)\geq 1}dx
 +\int_{0}^{\epsilon }dx\\ &\leq \epsilon +\int_{0}^{\epsilon
}f(tx)^{p(tx)}\chi _{f(tx)\geq 1}
=\epsilon +\frac{1}{t}\int_{0}^{t\epsilon }f(u)^{p(u)}du.
\end{align*}%
Whence,
\begin{equation*}
\int_{0}^{\epsilon }f(tx)^{p(x)}\leq \frac{1}{t}+\epsilon \leq \frac{2}{t},%
\hspace{0.5cm} 0<t<1.
\end{equation*}%
This implies
\begin{equation*}
\int_{0}^{\epsilon }\left( \frac{f(tx)}{t^{-\frac{1}{p^{-}}}2^{\frac{1}{p^{-}%
}}}\right) ^{p(x)}dx\leq 1,\,\,\,0<t<1.
\end{equation*}%
Therefore and using the definition of $p(.)$ -norms, we get
\begin{equation}
\left\Vert f(\cdot\,t)\right\Vert _{p(.);\,(0,\epsilon )}\leq 2^{\frac{1}{p^{-}}%
}t^{-\frac{1}{p^{-}}},\,\,\,0<t<1.  \label{e:3.9}
\end{equation}%
Using \eqref{e:3.9} and \eqref{E:3.7} for the first summand in \eqref{e:3.4}
we have the estimate
\begin{equation}  \label{e:3.8}
i_{1}\leq 2^{%
\frac{1}{p^{-}}}\int_{0}^{1}t^{-\frac{1}{p^{-}}}dt\leq \frac{p^{-}}{p^{-}-1}%
2^{\frac{1}{p^{-}}}.
\end{equation}

Now we shall estimate the term $\left\Vert \frac{Hf(.)}{.}\right\Vert
_{p(.);\,(\epsilon ,1)}.$ For $x\in (\epsilon ,1)$ \ using Young's inequality%
$,$ we get
\begin{align*}
& \int_{0}^{1}f(tx)dt\leq \int_{0}^{1}\frac{{f(tx)}^{p(tx)}}{p(tx)}%
dt+\int_{0}^{1}\frac{dt}{p^{\prime }(tx)} \\
& \frac{1}{xp^{-}}\int_{0}^{x}f(u)^{p(u)}du+\frac{p^{-}-1}{p^{-}}\leq \frac{1%
}{\epsilon p^{-}}+\frac{1}{(p^{+})^{\prime }}\leq 1+\frac{1}{\epsilon}.
\end{align*}%
Therefore,
\begin{align*}
i_{2}=& \left\Vert \frac{Hf(.)}{.}\right\Vert _{p(.);\,(\epsilon
,1)}=\left\Vert \int_{0}^{1}f(.\,t)dt\right\Vert _{p(.);\,(\epsilon ,1)} \\
& \leq \left( \frac{1}{\epsilon }+1\right)
\left\Vert 1\right\Vert _{p(.);\,(\epsilon ,1)}\leq C.
\end{align*}%
Inserting this estimate and \eqref{e:3.8} in \eqref{e:3.4} we complete the
proof of Theorem \ref{t:2.1}.

\section{ Proof of Theorem \ref{t:2.2}.}

To prove Theorem \ref{t:2.2} we need several lemmas.
\begin{lemma}\label{l:3.1}
Let $p:(0,1)\to[1,\infty)$ be a monotone nondecreasing function such that $p(1)<\infty$ and the condition
\eqref{E:1.4}
is satisfied. Then there exists a constant $C_{1}>0$\, depending on
 $C, \, p(1)$  such that the condition
 \begin{equation}\label{e:1.3}
 \left \vert \frac{1}{p'(2x)}-\frac{1}{p'(x)}\right \vert \ln \frac{1}{x}\leq C_{1}
 \end{equation}
 is satisfied.
\end{lemma}
\begin{proof}
From \eqref{E:1.4}
it follows that
$$\int_{2a}^{4a} \left(x^{-\frac{1}{p'(x)}} a^{\frac{1}{p'(a)}}
\right)^{p(x)} \frac{dx}{x} \leq C. $$
Since $\frac{1}{p'(x)}$ is monotone nondecreasing, we have
$$\int_{2a}^{4a} \left((4a)^{-\frac{1}{p'(2a)}} a^{\frac{1}{p'(a)}}
\right)^{p(x)} \frac{dx}{x} \leq C. $$
Suppose $ a^{\frac{1}{p'(a)}}(4a)^{-\frac{1}{p'(2a)}}$ is greeter then 1.
Then
$$C\geq \left((4a)^{-\frac{1}{p'(2a)}} a^{\frac{1}{p'(a)}}
\right)^{p(0)}\ln2 \geq 4^{1-p(0)}\ln 2 \, a^{\frac{1}{p'(a)}-\frac{1}{p'(2a)}}.$$
Whence,
$$\left(\frac{1}{a}\right)^{\frac{1}{p'(2a)}-\frac{1}{p'(a)}}\leq 1+\frac{C4^{p(1)-1}}{\ln 2}$$
or $$\left( \frac{1}{p'(2a)}-\frac{1}{p'(a)}\right)\ln \frac{1}{a}\leq \ln \left(\frac{C4^{p(1)-1}}{\ln 2}+1\right)$$
This completes the proof of Lemma \ref{l:3.1} with constant $C_{1}=\ln \left(\frac{C4^{p(1)-1}}{\ln 2}+1\right).$
\end{proof}

\begin{lemma} \label{l:3.2}
Let $p:(0,1)\to [1,\infty)$ be a nondecreasing function satisfying the condition \eqref{E:1.4} and $p(1)<\infty$. Then there exists
a constant $C_{1}>0$
depending on $C$ and $p(0)$ such that for any $\frac{x}{2}\leq y \leq 2x, \, \, 0<x<\frac{1}{4}$  the estimate
\begin{equation}\label{e:4.2}
\frac {1}{C_{1}}\phi (x)\leq \phi(y)\leq C_{1}\phi (x)
\end{equation}
holds, where the function $\phi(t)=t^{-\frac{1}{p'(t)}}.$

\end{lemma}%
\begin{proof}
Since $\frac{1}{p'} $ is nondecreasing it follows from Lemma \ref{l:3.1} that
\begin{align*}
\phi (y)&\leq \left( \frac{x}{2}\right)^{-\frac{1}{p'(y)}}
\leq \left(\frac{1}{x}\right)^{\frac{1}{p'(2x)}-\frac{1}{p'(x)}} x^{-\frac{1}{p'(x)}}2^{\frac{1}{p'(1)}}\\
&\leq 2\left(   C4^{p(1)-1}+1\right)\phi(x).
\end{align*}
By the same way,

\begin{align*}
\phi (x)&\leq \left( \frac{y}{2}\right)^{-\frac{1}{p'(x)}}
\leq \left(\frac{1}{y}\right)^{\frac{1}{p'(2y)}-\frac{1}{p'(y)}} y^{-\frac{1}{p'(y)}}2^{\frac{1}{p'(1)}}\\
&\leq 2\left(   C4^{p(1)-1}+1\right)\phi(y).
\end{align*}
Therefore, \eqref{e:4.2} is satisfied by the constant $2\left(   C4^{p(1)-1}+1\right).$
\end{proof}

\begin{lemma}\label{l:3.3}
Let $p:(0,1)\to [1,\infty )$ be a nondecreasing function such that $p(1)<\infty$ and the condition \eqref{E:2.2} is satisfied.
Then there exists a constant $C_{1}>0$ depending on $C$ such that the condition \eqref{e:1.3}
is satisfied.
\end{lemma}
\begin{proof}
Using \eqref{E:2.2} we have
\begin{equation*}
Ca^{-\frac{1}{p'(a)}}\geq \int_{2a}^{4a}x^{-\frac{1}{p'(x)}}\frac{dx}{x}\geq \left( \frac{1}{%
4a}\right) ^{\frac{1}{p^{\prime }(2a)}}\ln 2\geq 4^{-\frac{1}{p^{\prime }(1)}%
}\ln 2\left( \frac{1}{a}\right) ^{\frac{1}{p^{\prime }(2a)}};
\end{equation*}%
that is,
\begin{equation*}
\left( \frac{1}{a}\right) ^{\frac{1}{p^{\prime }(2a)}-\frac{1}{p^{\prime }(a)%
}}\leq \frac{4C}{\ln 2}.
\end{equation*}%
This proves \eqref{e:1.3} with constant $C_{1}=\ln\left(\frac{4C}{\ln 2}\right).$
\end{proof}
\bigskip
\begin{lemma}\label{l:3.4}
Let \ $p:(0,1)\to [1,\infty )$ be a
nondecreasing function satisfying the conditions \eqref{E:2.2} and $p(1)<\infty$. Then there exists a
constant $C_{1}$ such that
\begin{equation*}
\frac{1}{C_{1}}\phi (x)\leq \phi (y)\leq C_{1}\phi (x),
\end{equation*}%
for any $\frac{x}{2}$ $<y<2x,$ $0<x<\frac{1}{4},$
where the function $\phi(t)=t^{-\frac{1}{p'(t)}}.$
\end{lemma}%
\begin{proof}
To prove Lemma \ref{l:3.4} it suffice to apply Lemma \ref{l:3.3} as in Lemma \ref{l:3.2}.
\end{proof}

\begin{lemma}\label{l:3.5}
Let $p:(0,1)\to [1,\infty)$ be a nondecreasing function such that $p(1)<\infty$. Then the following two assertions are equivalent:
\begin{itemize}
\item[1)] The condition \eqref{E:2.2} is satisfied.
\item[2)] There exists an $\epsilon >0$ such that the
function $x^{\epsilon}\phi(x)$ is almost decreasing: there exists a $C_{1}>0 $ such that
\begin{equation}\label{E:2.7}
t_{2}^{\epsilon}\phi(t_{2})\leq C_{1} t_{1}^{\epsilon}\phi(t_{1}),  \, \, \, 0<t_{1}\leq t_{2}<1.
\vspace{0.5cm}
\end{equation}
Here the function $\phi(t)=t^{-\frac{1}{p'(t)}}.$
\end{itemize}

\end{lemma}%
\begin{proof}\textit{Proof of $1)\to 2).$}
  Denote $g(x)=\int_{x}^{1}\phi (t) \frac{dt}{t}.$
Then
\begin{equation*}
g'(x)=-\frac{\phi(x)}{x}, \hspace{1cm} 0<x<1.
\end{equation*}
Hence
\begin{equation*}
g(x) \leq-Cg'(x)x \hspace{1cm}\text{or} \hspace{1cm}
\frac{1}{C}\frac{1}{x} \leq\frac{-g'(x)}{g(x)}, \hspace{0.5cm}  0<x<1.
\end{equation*}
Integrating this inequality in $x$ over $(t_{1},t_{2}),$  we get
\begin{equation*}
\ln \frac{g(t_{1})}{g(t_{2})}\geq \frac{1}{C}\ln \frac{t_{1}}{t_{2}} \hspace{1cm}    \text {or} \hspace{1cm}
g(t_{2})t_{2}^{\frac{1}{C}}\leq g(t_{1})t_{1}^{\frac{1}{C}}.
\end{equation*}
Since $$ g(t_{2})=\int_{t_{2}}^{1}\phi (x) \frac{dx}{x}\geq \int_{t_{2}}^{2t_{2}}\phi (x) \frac{dx}{x}\geq \frac{1}{C}\phi(t_{2})\ln 2,$$ using \eqref{E:2.2} and assertion of Lemma \ref{l:3.4} we get
$$\frac{\ln 2}{C}\phi(t_{2})t_{2}^{\frac{1}{C}}\leq C \phi(t_{1})t_{1}^{\frac{1}{C}}
$$
Therefore, \eqref{E:2.7} is satisfied with $\epsilon=\frac{1}{C}, \, \, C_{1}=C^{2}.$ \\[5pt]

\emph{Proof of $2)\to 1).$} Estimating directly, we have
\begin{align*}
\int_{a}^{1}\phi (x)\frac{dx}{x}&=\int_{a}^{1}x^{\epsilon}\phi (x)\frac{dx}{x^{1+\epsilon}}\leq C
\int_{a}^{1}a^{\epsilon}\phi (a)\frac{dx}{x^{1+\epsilon}}\\
&=Ca^{\epsilon}\phi(a)\int_{a}^{1}\frac{dx}{x^{1+\epsilon}}=\frac{C}{\epsilon}\phi(a).
\end{align*}
The inequality \eqref{E:2.2} has been proved.
\end{proof}

\bigskip

\begin{lemma} \label{l:3.6}
Let $p:(0,1) \to [1,\infty )$ be \ nondecreasing function such that $p(1)<\infty$.
Then the condition \eqref{E:1.4} is necessary for the inequality \eqref{E:1.1}
to hold.
\end{lemma}%
\begin{proof}
Let $a\in (0,1)$ be a fixed number. Put a test function
\begin{equation*}
f_{0}(x)=x^{-\frac{1}{p(x)}}\chi _{(\frac{a}{2},a)}(x),\hspace{0.5cm} 0<x<1,
\end{equation*}%
into the inequality \eqref{E:1.1}. Then
\begin{equation*}
I_{p(.)}\left( f_{0}\right) =\int_{\frac{a}{2}}^{a}\frac{dx}{x}=\ln
2\leq 1,
\end{equation*}%
therefore, $\left\Vert f_{0}\right\Vert _{p(.)}\leq 1.$ Hence $\left\Vert
\frac{Hf_{0}}{x}\right\Vert _{p(.);(0,1)}\leq C.$ This implies that $%
I_{p(.)}\left( \frac{Hf_{0}}{x}\right) \leq C_{2},$ whence%
\begin{align*}
C_{2} &\geq \int_{a}^{1}\left( \int_{\frac{a}{2}}^{a}t^{-%
\frac{1}{p(t)}}dt\right) ^{p(x)}x^{-p(x)}dx\geq \int_{a}^{1}\left(
\frac{a}{2}a^{-\frac{1}{p(a)}}\right) ^{p(x)}x^{-p(x)}dx \\
&\geq 2^{-p^{+}}\int_{a}^{1}\left( a^{\frac{1}{p'(a)}}x^{-%
\frac{1}{p'(x)}}\right) ^{p(x)}\frac{dx}{x}.
\end{align*}%
Hence
\begin{equation*}
\int_{a}^{1}\left( a^{\frac{1}{p^{\prime }(a)}}x^{-\frac{1}{%
p^{\prime }(x)}}\right) ^{p(x)}\frac{dx}{x}\leq C_{3}.
\end{equation*}%
\end{proof}

\begin{lemma}\label{l:3.7}
Let $p:(0,1)\to[1,\infty)$ be a nondecreasing function satisfying the conditions $p(1)<\infty$ and \eqref{E:1.4}.
Then the function $\phi(x)=x^{-\frac{1}{p'(x)}}$ is almost decreasing;
that is for any $0<t_{1}\leq t_{2}<1$ we have $$ \phi(t_{2})\leq C \phi(t_{1})$$
\end{lemma}%
\begin{proof}
Put $t_{1}=a.$ Let $2^{k-1}a\leq t_{2}<2^{k}a,$ \hspace{0.2cm}$k\in N.$ Then
using \eqref{E:1.4} and Lemma \ref{l:3.2}, we have%
\begin{align*}
C &\geq \sum_{n=1}^{\infty }\int_{2^{n-1}a}^{2^{n}a}\left(
a^{\frac{1}{p'(a)}}x^{-\frac{1}{p'(x)}}\right) ^{p(x)}%
\frac{dx}{x} \\
&\geq \sum_{n=1, n\in \textbf{N'}}^{\infty }\int_{2^{n-1}a}^{2^{n}a}\left(
a^{\frac{1}{p'(a)}}\left( 2^{n}a\right) ^{-\frac{1}{p'(2^{n}a)}}\right) ^{p^{-}}\frac{dx}{x} \\
&+\sum_{n=1,n\in \textbf{N''}}^{\infty }\int_{2^{n-1}a}^{2^{n}a}\left( a^{%
\frac{1}{p'(a)}}\left( 2^{n}a\right) ^{-\frac{1}{p'(2^{n}a)%
}}\right)^{p^{+}}\frac{dx}{x},
\end{align*}
where $\sum_{n=1, n\in \textbf{N'}}^{\infty }(...)$ \ means summing over $n\in \textbf{N}$
such that  $a^{\frac{1}{p'(a)}}\left( 2^{n}a\right) ^{-\frac{1}{%
p'(2^{n}a)}}\geq 1$ and $\sum_{n=1, n\in \textbf{N''}}^{\infty }(...)$
means summing over $n\in \textbf{N}$ such
that $a^{\frac{1}{p'(a)}}\left(
2^{n}a\right)^{-\frac{1}{p'(2^{n}a)}}\leq 1.$
Therefore,%
\begin{equation}\label{e:4.11}
a^{\frac{1}{p^{\prime }(a)}}\left( 2^{n}a\right) ^{-\frac{1}{p^{\prime
}(2^{n}a)}}\leq 1+\frac{C}{C_{1}\ln 2},\hspace{0.5cm} n\in \textbf{N}.
\end{equation}%
Further using the Lemma \ref{l:3.2}, we deduce from \eqref{e:4.11}
\begin{equation*}
a^{\frac{1}{p^{\prime }(a)}}\left( 2^{k}a\right) ^{-\frac{1}{p^{\prime
}(2^{k}a)}}\leq C_{3},
\end{equation*}%
hence by using Lemma \ref{l:3.2}, we have
\begin{equation*}
a^{\frac{1}{p'(a)}}t_{2}^{-\frac{1}{p'(t_{2})}}\leq C_{4}.
\end{equation*}%

This completes the proof of Lemma \ref{l:3.7}.
\end{proof}

\begin{lemma}\label{l:3.8}
Let  $p:(0,1)\to [1,\infty )$ be a nondecreasing
function satisfying the conditions $p(1)<\infty$  and  \eqref{E:1.4}. Then the condition \eqref{E:2.2} is satisfied,
moreover, the function $x^{-\frac{1}{p^{\prime }(x)}+\epsilon }$ is almost decreasing by some $\epsilon>0$.
\end{lemma}%
\begin{proof}
Using \eqref{E:1.4} and Lemma \ref{l:3.7} we have the estimates%
\begin{align*}
C &\geq \int_{a}^{1}\left( a^{\frac{1}{p^{\prime }(a)}}x^{-\frac{1}{%
p^{\prime }(x)}}\right) ^{p(x)}\frac{dx}{x} \\
&\geq C_{4}^{p^{-}}\int_{a}^{1}\left( \frac{1}{C_{4}}a^{\frac{1}{%
p^{\prime }(a)}}x^{-\frac{1}{p^{\prime }(x)}}\right) ^{p(x)}\frac{dx}{x} \\
&C_{4}^{p^{-}-p^{+}}\int_{a}^{1}\left( a^{\frac{1}{p^{\prime }(a)}%
}x^{-\frac{1}{p^{\prime }(x)}}\right) ^{p^{+}}\frac{dx}{x}.
\end{align*}%
This implies%
\begin{equation}\label{e:4.13}
\int_{a}^{1}x^{-\frac{p^{+}}{p^{\prime }(x)}}\frac{dx}{x}\leq
C_{4}^{p(1)-1}a^{-\frac{p^{+}}{p^{\prime }(a)}},\hspace{0.5cm} 0<a<1.
\end{equation}%
Applying the approach of Lemmas \ref{l:3.3} and \ref{l:3.7}, we find the function $x^{-\frac{p^{+}%
}{p^{\prime }(x)}}$ is almost decreasing and satisfies the condition \eqref{e:4.13}. It
follows from the Bari-Stechkin theorem \cite{[1]} (see, also \cite{[90]}) that there exists an $\epsilon >0$
such that the function $\ x^{-\frac{p^{+}}{p^{\prime }(x)}+\epsilon }$ is almost decreasing. This implies the function $x^{-\frac{1}{p^{\prime }(x)}%
+\epsilon _{1}}$ is almost decreasing. Again using Bari-Stechkin result \cite{[1]} we
deduce the function $x^{-\frac{1}{p^{\prime }(x)}}$ satisfies the condition \eqref{E:2.2}%

Hence we have proved that (by using Lemmas \ref{l:3.6} and \ref{l:3.8} for the inequality \eqref{E:1.1}
to hold it is necessary the condition \eqref{E:2.2}. Let us prove that the condition
\eqref{E:2.2} is also sufficient for \eqref{E:1.1}.
\end{proof}
\begin{remark}
It follows from Lemma \ref{l:3.8} that the condition \eqref{E:2.2} for nondecreasing $p:(0,1)\to [1,\infty)$ implies $p(0)>1$.
Hence the condition $p(0)>1$
is necessary (but not sufficient) for the inequality \eqref{E:1.1} to hold.
\end{remark}
\begin{lemma}\label{l:3.9}
Let $p:(0,1)\to [1,\infty )$ be a nondecreasing
function such that the conditions \eqref{E:2.2} and $p(1)<\infty$ is satisfied. Then the inequality
\eqref{E:1.1} holds.
\end{lemma}
\begin{proof}
Using Lemma \ref{l:3.3} we infer that the function $x^{-\frac{1}{p^{\prime }(x)}%
}$ is almost decreasing. Further, according to Lemma \ref{l:3.5} the condition \eqref{E:2.2} implies that the
function $x^{-\frac{1}{p^{\prime }(x)}+\epsilon }$ is almost decreasing by some
$\epsilon >0.$

Let us prove sufficiency of condition \eqref{E:2.2}. It suffices to consider the case when function $f(x)\geq 0$ is a measurable function such that $\left\Vert
f\right\Vert _{L^{p(.)}(0,1)}\leq 1$ (see, \cite{[3]}). Then
$I_{p(.)}\left( f\right) \leq 1.$
In order to prove Lemma \ref{l:3.9} we have to prove
$\left\Vert x^{-1}Hf\right\Vert _{L^{p(.)}(0,1)}\leq C_{1}.$
We shall derive this inequality from the estimate
$I_{p(.)}\left( x^{-1}Hf\right) \leq C_{2}.
$

By Minkowski inequality, for $L^{p(.)}$ norms, we get the inequalities%
\begin{equation*}
\left\Vert x^{-1}Hf\right\Vert _{L^{p(.)}(0,1)}\leq \left\Vert x^{-\frac{1}{p(x)}-\frac{1}{p%
(x)}}\sum_{n=0}^{\infty
}\int_{2^{-n-1}x}^{2^{-n}x}f(t)dt\right\Vert _{L^{p(.)}(0,1)}
\end{equation*}%
\begin{equation}
\leq \sum_{n=0}^{\infty }\left\Vert x^{-\frac{1}{p(x)}-\frac{1}{p
(x)}}\int_{2^{-n-1}x}^{2^{-n}x}f(t)dt\right\Vert _{L^{p(.)}(0,1
)}  \label{E:3.4}
\end{equation}%
Denote $B_{x,n}=(2^{-n-1}x,2^{-n}x]$ and $p_{x,n}=\inf \{p(t):$ $t\in
B_{x,n}\};n=1,2,...$. Put $\varphi (t)=t^{\frac{1}{p
(t)}}.$ Since the condition \eqref{E:2.2} holds, it follows from Lemma \ref{l:3.8} that
there exists an $\epsilon \in (0,1)$
such that%
\begin{equation}
\frac{\varphi (s)}{s^{\epsilon }}\leq C\frac{\varphi (r)}{r^{\epsilon }},\text{ }%
0<s<r<1.  \label{E:3.5}
\end{equation}%
Then by \eqref{E:3.5} we have%
\begin{equation}
\frac{\varphi (t)}{t^{\epsilon }}\leq C\frac{\varphi (x)}{x^{\epsilon }},
\label{E:3.6}
\end{equation}%
where $t$ is a point in $B_{x,n},0<x<1 $ and the constant $C$ \ does
not depend on $n.$

By using inequality \eqref{E:3.6} and $2^{-n-1}x<t<2^{-n}x$ we have the
estimates%
\begin{equation*}
t^{\frac{1}{p'
(t)}}=t^{\epsilon }t^{\frac{1}{p'
(t)}-\epsilon }\leq Ct^{\epsilon }x^{\frac{1}{p'
(x)}-\epsilon}\leq C2^{-n\epsilon }x^{\frac{1}{p'
(x)}}.
\end{equation*}%
Hence%
\begin{equation*}
x^{-\frac{1}{p'
(x)}}\leq C2^{-n\epsilon }t^{-\frac{1}{p'
(t)}.}
\end{equation*}%
Therefore, and due to Holder's inequality, for $x\in B(0,1),$ we get%
\begin{equation*}
x^{-\frac{1}{p(x)}-\frac{1}{p'
(x)}}\sum_{n=0}^{\infty }\int_{2^{-n-1}x}^{2^{-n}x}f(t)dt
\end{equation*}%
\begin{equation*}
\leq C2^{-n\epsilon}x^{-\frac{1}{p(x)}}t^{-\frac{1}{p'
(t)}}\int_{2^{-n-1}x}^{2^{-n}x}f(t)dt
\end{equation*}%
\begin{equation}\label{E:3.9}
\leq C2^{-n\epsilon }x^{-\frac{1}{p(x)}}t^{-\frac{1}{p'
(t)}}\left( \int_{2^{-n-1}x}^{2^{-n}x}f(t)^{p_{x,n}^{-}}dt\right)^
\frac{1}{p_{x,n}^{-}}\left( 2^{-n}x\right) ^{\frac{1}{\left(
p_{x,n}^{-}\right)'}}
\end{equation}%
It follows from Lemma \ref{l:3.2} that%
\begin{equation}\label{E:3.10}
\left( 2^{-n}x \right) ^{\frac{1}{\left( p_{x,n}^{-}\right)'}}
\leq 2^{-\frac{1}{\left( p_{x,n}^{-}\right)'}}
t^{\frac{1}{\left( p_{x,n}^{-}\right)'}}
\leq C_{1}t^{\frac{1}{p'(t)}},
\end{equation}%
where $C$ depends only $p.$

Combining \eqref{E:3.9} and \eqref{E:3.10} we get%
\begin{equation}
x^{-\frac{1}{p(x)}-\frac{1}{p'
(x)}}\sum_{n=0}^{\infty
}\int_{2^{-n-1}x}^{2^{-n}x}f(t)dt\leq C2^{-n\epsilon }x^{-\frac{1}{p(x)%
}}\left( \int
_{2^{-n-1}x}^{2^{-n}x}f(t)^{p_{x,n}^{-}}dt\right) ^{%
\frac{1}{p_{x,n}^{-}}}  \label{E:3.11}
\end{equation}%
where $0<x<1 ,$ $n=1,2,...$and the constant $C_{2}$ does not depend on $%
n,x.$

Simultaneously,%
\begin{equation*}
\int_{2^{-n-1}x}^{2^{-n}x}f(t)^{p_{x,n}^{-}}dt\leq
\int_{2^{-n-1}x}^{2^{-n}x}f(t)^{p(t)}\chi _{\left\{ f(t)\geq
1\right\} }dt+\int_{2^{-n-1}x}^{2^{-n}x}dt\leq 1+2^{-n} \leq
C_{3}.
\end{equation*}%
By the last inequality and \eqref{E:3.11}, we have%
\begin{equation*}
I_{p(.)}\left( x^{-\frac{1}{p(x)}-\frac{1}{p'
(x)}}\int_{2^{-n-1}x}^{2^{-n}x}f(t)dt\right) \leq C_{4}2^{-n\epsilon
p^{-}}\int_{0}^{1}x^{-1}\left(
\int_{2^{-n-1}x}^{2^{-n}x}f(t)^{p_{x,n}^{-}}dt\right) ^{\frac{p(x)}{%
p_{x,n}^{-}}}dx
\end{equation*}%
\begin{equation*}
\leq C_{4}C_{3}^{\frac{p^{+}}{p^{-}}-1}2^{-n\epsilon
p^{-}}\int_{0}^{1}x^{-1}\left(
\int_{2^{-n-1}x}^{2^{-n}x}\left( f(t)^{p(t)}+1\right) dt\right) dx
\end{equation*}%
which, due to Fubini's theorem, yields%
\begin{equation*}
\leq C_{4}C_{3}^{\frac{p^{+}}{p^{-}}-1}2^{-n\epsilon p^{-}}\ln
2\int_{0}^{2^{-n}}\left(
\int_{2^{-n-1}x}^{2^{-n}x}x^{-1}dx\right) \left(
f(t)^{p(t)}+1\right) dt
\end{equation*}%
\begin{equation}
=C_{5}2^{-n\epsilon p^{-}}\ln 2\int\limits_{0}^{2^{-n}}\left(
f(t)^{p(t)}+1\right) dt\leq C_{6}2^{-n\epsilon p^{-}}.  \label{E:3.16}
\end{equation}%
Therefore,%
\begin{equation*}
\left\Vert x^{-\frac{1}{p(x)}-\frac{1}{p'
(x)}}\int_{2^{-n-1}x}^{2^{-n}x}f(t)dt\right\Vert _{L^{p(.)}(0,1
)}\leq C2^{-\frac{n\epsilon p^{-}}{p^{+}}}
\end{equation*}%
By \eqref{E:3.16} and \eqref{E:3.4}, we get%
\begin{equation*}
\left\Vert x^{-1}Hf\right\Vert _{L^{p(.)}(0,1)}\leq C\sum_{n=0}^{\infty }2^{-\frac{n\epsilon p^{-}}{p^{+}}}\leq
C_{1}.
\end{equation*}

This completes the proof of Lemma \ref{l:3.9}.
\end{proof}

\textbf{Proof of Theorem \ref{t:2.2}}. Let 5) be satisfied, that is the condition
\eqref{E:2.4}. Then by the definition,
\begin{equation*}
\int_{a}^{1}\left( \frac{x^{-1}}{\left\Vert (.)^{-1}\chi
_{\{a,1\}}(.)\right\Vert _{p(.)}}\right) ^{p(x)}dx\leq 1.
\end{equation*}%
Therefore, and using \eqref{E:2.4}, we have%
\begin{equation*}
\int_{a}^{1}\left( \frac{x^{-1}}{Ca^{-\frac{1}{p^{\prime }(a)}}}\right)
^{p(x)}dx\leq 1
\end{equation*}%
or%
\begin{equation*}
\int_{a}^{1}\left( a^{\frac{1}{p^{\prime }(a)}}x^{-\frac{1}{p^{\prime }(x)}%
}\right) ^{p(x)}\frac{dx}{x}\leq C_{1}.
\end{equation*}%
This is the condition \eqref{E:1.4}, that is 4) of Theorem \ref{t:2.2}. Hence $5)\to 4)$ has been proved. According to Lemmas
\ref{l:3.6}, \ref{l:3.7}, \ref{l:3.8}, we have the implication $4)\to 2).$ The implication $2)\to 3)$
follows from Lemma \ref{l:3.5}. The
implication $3)\to 1)$ follows from Lemma \ref{l:3.9}. The implication $1)\to 4)$
is proved in Lemma \ref{l:3.6}.

The implication $3)\to 5)$ is direct: using the condition \eqref{e:2.5}
we have
\begin{align*}& \int_{a}^{1}\left(a^{\frac{1}{p'(a)}}x^{-\frac{1}{p'(x)}}\right)^{p(x)}\frac{dx}{x} \leq \int _{a}^{1}\left(C\left(\frac{a}{x}\right)^\epsilon \right)^{p(x)}\frac{dx}{x}\\ &=\int_{1}^{\frac{1}{a}}C^{p(1)}\left(\frac{1}{t}\right)^{\epsilon p(at)}\frac{dt}{t}\leq C^{p(1)}\int _{1}^{\infty}\frac{dt}{t^{1+\epsilon p(0)}}<C_{2}.
\end{align*}
Rewriting the last inequality, we have
$$ \int_{a}^{1}\left(\frac{x^{-1}}{C_{2}^{\frac{1}{p(1)}}a^{-\frac{1}{p'(a)}}}\right)^{p(x)}dx\leq 1,$$
therefore, the condition \eqref{E:2.4} is satisfied.

This completes the proof of Theorem \ref{t:2.2}.

If the exponent function $p$ in Theorem \ref{t:2.2} is nondecreasing on not all the interval $(0,1)$ but so is only near the origin the following assertion holds.

\begin{remark}
 Let a measurable function $p:[0,1] \to (1,\infty)$ be nondecreasing on some interval $(0,\delta), \, 0<\delta<1$ and $p^{+}<\infty$; then the following statements are
equivalent:
\begin{itemize}
\item[a)] There exists a constant $C>0$ such that the inequality
\begin{equation}
\ \ \left\Vert x^{-1}Hf\right\Vert _{L^{p\left( .\right) }\left( 0,1\right)
}\leq C\left\Vert f\right\Vert _{L^{p\left( .\right) }\left( 0,1\right) }
\label{E:2.1}
\end{equation}
holds for any positive measurable function $f.$

\item[b)] The condition
\begin{equation}  \label{E:2.2}
\int_{a}^{\delta} x^{-\frac{1}{p^{\prime }(x)}}\frac{dx}{x}\leq Ca^{-\frac{1}{%
p^{\prime }(a)}}, \, 0<a<\delta
\end{equation}
is satisfied.

\item[c)] There exists an $\epsilon>0$ such that the function $x^{-\frac{1}{p^{\prime }(x)}+\epsilon}$
is almost decreasing:
\begin{equation}  \label{e:2.5}
t_{2}^{-\frac{1}{p^{\prime }(t_{2})}+\epsilon}\leq C t_{1}^{-\frac{1}{p^{\prime }(t_{1})}+\epsilon} \hspace{0.5cm}\text{as} \hspace{0.5cm}
0<t_{1}\leq t_{2}<\delta
\end{equation}
\item[d)] The condition
\begin{equation}  \label{E:1.4}
\int_{a}^{\delta}\left( a^{\frac{1}{p^{\prime }(a)}} x^{-\frac{1}{p^{\prime }(x)}%
}\right) ^{p(x)}\frac{dx}{x}\leq C,\text{ \ }0<a<\delta
\end{equation}
is satisfied.
\item[e)] The condition
\begin{equation}  \label{E:2.4}
\Vert x^{-1} \Vert_{p(.);(a,\delta)}\leq Ca^{-\frac{1}{p^{\prime }(a)}}, \, \, \,
0<a<\delta.
\end{equation}
is satisfied.
\end{itemize}
\end{remark}


\vspace{1cm}


\begin{thebibliography}{99}
\bibitem{[1]} N.K. Bari and S.B. Stechkin. "Best approximations and
differential properties of two conjugate functions" (in Russian). \textit{%
Proceedings of Moscow Mathematical Society,} 5:483-522, 1956.

\bibitem{[2]} D. Cruz.-Uribe, SFO and F. I. Mamedov, "On a general weighted
Hardy type inequality in the variable exponent Lebesgue spaces,"\textit{%
Revista Matematica Complutense, } vol. 25, no. 2, pp. 335-367, 2012.

\bibitem{[3]} L. Diening, P. Harjulehto, P. Hasto and M. Ruzicka,: "Lebesgue
and Sobolev Spaces with Variable Exponents," Lecture Notes in Mathematics,
vol 2017, Springer, Heidelberg, Germany, 2011.

\bibitem{[4]} L. Diening and S. Samko,"Hardy inequality in variable exponent
Lebesgue spaces," \textit{Fractional Calculus \& Applied Analysis,} vol. 10,
no 1, pp. 1-17, 2007

\bibitem{[5]} D. E. Edmunds, V. Kokilashvili and A. Meskhi, "On the
boundedness and compactness of the weighted Hardy operators in spaces,"%
\textit{\ Georgian Mathematical Journal}, vol. 12, no. 1, pp. 27-44, 2005.

\bibitem{[6]} X. L. Fan and D. Zhao, "On the spaces $L^{p(x)}(\Omega )$ and $%
W^{m,p(x)}(\Omega )$," \textit{Journal of Mathematical Analysis and
Applications,} vol. 263, no. 2, pp. 424-446, 2001.

\bibitem{[7]} A. Harman, "On necessary condition for the variable exponent
Hardy inequality," \textit{Journal of Function Spaces and Applications,}
vol. 2012, Article ID 385925, 6 pages, doi:10.1155/2012/385925

\bibitem{[8]} P. Harjulehto, P.Hasto and M. Koskinoja, "Hardy`s inequality
in variable exponent Sobolev spaces," \textit{Georgian Mathematical Journal,}
vol. 12, no. 3, pp. 431-442, 2005.

\bibitem{[9]} A. Harman and F.I. Mamedov, "On boundedness of weighted Hardy
operator in $L^{p(.)}$ and regularity condition," \textit{Journal of
Inequalities and Applications,} vol. 2010, Article ID 837951, 14 pages, 2010.


\bibitem{[90]} V. Kokilashvili, S. Samko and N. Samko, "The Maximal Operator in Weighted Variable Spaces $L^{p(.)}$" \textit{J. Function Spaces Appl.},  vol. 5, no 3, pp. 299-317, 2007.



\bibitem{[10]} F. I. Mamedov, "On Hardy type inequality in variable
exponent Lebesgue space $L^{p(.)}(0,1)$," \ \textit{Azerbaijan Journal of
Mathematics,} vol. 2, no. 1, pp. 90--99, 2012.

\bibitem{[11]} F.I. Mamedov and A. Harman, "On a weighted inequality of
Hardy type in spaces $L^{p(.)}$," \textit{\ Journal of Mathematical Analysis
and Applications,} vol. 353, no. 2, pp. 521-530, 2009.

\bibitem{[12]} F.I. Mamedov and A. Harman, "On a Hardy type general weighted
inequality in spaces $L^{p(.)}$," \textit{Integral Equations and Operator
Theory,} vol. 66, no. 4, pp. 565-592, 2010.

\bibitem{[13]} F.I. Mamedov and Y. Zeren, "On equivalent
conditions for the general weighted Hardy type inequality in space $L^{p(.)},$" \textit{Zeitschrift fur Analysis und ihre Anwendungen,} vol. 31, no 1, pp. 55-74, 2012

\bibitem{[14]} R. Mashiyev, B. Cekic, F. I. Mamedov and S. Ogrash, "Hardy`s
inequality in power-type weighted $L^{p(.)}$ spaces," \textit{Journal of
Mathematical Analysis and Applications,} vol. 334, no. 1, pp. 289-298, 2007.

\bibitem{[15]} H. Rafeiro and S. G. Samko, "Hardy inequality in variable
Lebesgue spaces," \textit{Annales Academiae Scientiarium Fennicae,} vol. 34,
no. 1, pp. 279-289, 2009.

\bibitem{[16]} S. G. Samko, "Hardy inequality in the generalized Lebesgue
spaces," \textit{Fractional Calculus \& Applied Analysis,} vol. 6, no. 4,
pp. 355-362, 2003.

\bibitem{[17]} S. G. Samko, "Convolution type operators in $L^{p(.)} $,"
\textit{Integral Transforms and Special Functions,} vol. 7, pp. 123-144,
1998.
\end{thebibliography}
\end{document}